\documentclass [12pt,a4paper,reqno]{amsart}
\input amssymb.sty
\newtheorem{thm}{Theorem}[section] 

\newtheorem{lem}[thm]{Lemma}

\newtheorem{proposition}[thm]{Proposition}

\newtheorem{rem}[thm]{Remark}

\def\scong{{\scriptstyle\|}\lower.2ex\hbox{$\wr$}}
\def\Z{{\mathbb Z}}

\def\Q{{\mathbb Q}}

\def\deg{\mathop{\rm deg}\nolimits}

\def\Gal{\mathop{\rm Gal}\nolimits}

\def\deg{\mathop{\rm deg}\nolimits}

\def\Schur{\mathop{\rm index}\nolimits}

\def\rtimes{\mathop{\times\!\!{\raise.2ex\hbox{$\scriptscriptstyle|$}}}
    \nolimits}
\def\proof{\noindent{\it Proof.}\quad}

\outer\def\Demo #1. #2\par{\medbreak\noindent {\it#1.\enspace}
    {\rm#2}\par\ifdim\lastskip<\medskipamount\removelastskip
    \penalty55\medskip\fi}

\def\index{\mathop{\rm index}}
\def\hangbox to #1 #2{\vskip1pt\hangindent #1\noindent \hbox to #1{#2}$\!\!$}

\begin{document}
\title[Simultaneous embeddings of finite dimensional division algebras] {Simultaneous embeddings of finite dimensional division algebras}

 \author[Louis Rowen]{Louis Rowen}
\address{Department of Mathematics, Bar-Ilan University, Ramat-Gan 52900,
Israel} \email{rowen@macs.biu.ac.il}
 \author[David J Saltman]{David J Saltman}
\address{CCR-Princeton, 805 Bunn Dr.,Princeton NJ 08540}
\email{saltman@idaccr.org}
\maketitle \

 A celebrated theorem of P.M.~Cohn \cite{C} says that for
any two division rings (not necessarily finite dimensional) over a field $F$,
their amalgamated product over $F$ is a domain which can be
embedded in a division ring.  Note that even with the two initial
division rings begin finite dimensional over their centers,
the resulting division ring is {\bf never} finite dimensional
over its center. Perhaps this led Lance Small to ask the following
question. We say $D/F$ is a division algebra when $D$ is a division
ring finite dimensional over its center $F$. Assume $F_1$ and $F_2$
are fields with the same characteristic. Small asked
whether any two division algebras $D_1/F_1$ and $D_2/F_2$
can be embedded in some third division algebra $E/F$.

We start with a surprisingly straightforward
counterexample in the next section, but then show that a
positive solution exists for division algebras finitely generated
over a common subfield which is either algebraically closed or the
prime subfield (Theorem~\ref{thm7}).

\section{A counterexample}

Suppose, first of all, that $D_1/F_1$ is a f.d.~ division algebra
embedded in the division algebra $E/F$, so $F_1$ and $F$ share the same
prime subfield $P$.
There is a tower of subalgebras $F
\subseteq F_1F \subset D_1F \subset E$, where $F_1F$ must be an
amalgamation of $F_1$ and~$F$ (meaning that it is the field of
fractions of an image of the tensor product $F_1 \otimes_P F$).

Getting more specific, suppose $p_1 \ne p_2$ are primes.
Let $G$ be the infinite cyclic profinite $p_2$-group,
  the inverse limit of all $\mathbb Z/{p_2}^n\mathbb Z$. Take $F_1/\mathbb Q$
Galois with group $G$. (For example, $F_1$ could be contained in
the infinite extension of $\mathbb Q$ obtained by adjoining all
$p_2^n$ roots of 1.) Note that $G$ has no finite subgroups. The
field extension $F_1F/F$  has Galois  group a subgroup of $G$, and
must be finite dimensional (being inside $E/F$), and so must be
trivial. That is, $F_1F = F$, implying $F_1 \subseteq F$.

Next, take $D_1'/\mathbb Q$ of degree $p_1$  and let  $D_1 = D_1'
\otimes_{\mathbb Q} F_1$, a division algebra since $F_1/\mathbb Q$
is a pro-{$p_2$} extension, and let $D_2/\Q$ be any division algebra
split by~$F_1$. For example, there is a cyclic degree
$p_2$ extension $L/\Q$ such that $L \subset F_1$. By class field theory
there is a degree $p_2$ division algebra with maximal subfield $L$.

\begin{proposition} There is no division algebra $E/F$ containing
both $D_1$ and $D_2$.
\end{proposition}

\begin{proof}
If $E/F$ contained both $D_1$ and $D_2$, then $F_1 \subseteq F$ is
central, by the above paragraph, so $E$ contains $D_2 F_1$
which is a division algebra but also a homomorphic image of the
split algebra $D_2 \otimes _{\Q} F_1$ and thus is commutative, a
contradiction.
\end{proof}

The rationale for this example is that the centers are
incompatible in some sense.

\section{Positive results}

\begin{rem} Since every division algebra is a tensor product of division algebras of prime power degree, it is natural to ask that if $D_1/F_1$ and $D_2/F_2$ are division algebras over
respective degrees
$p^{t_1}$ and $p^{t_2}$, then can $D_1/F_1$ and $D_2/F_2$ be embedded into
a single division algebra $E/F$ of degree $p^t$,
and is there a bound for $t$ in terms of $t_1$ and $t_2$? What
would be the best bound?
\end{rem} 

We approach the problem via \cite{Sa}. First let us fix some
notation. We write $\index(D)$ for the (Schur) index of $D$.
Fixing $r>1,$ let $UD(F,n)/Z(F,n)$ denote the generic division algebra of
degree $n$ over $F$ in $r$ indeterminates. We write $Z$ for $ Z(F,n)$.
When $L/F$ is a cyclic Galois extension of dimension $n$ and $a\in
F,$ $\Delta(L/F,a)/F$ denotes the $F$-central cyclic algebra having
maximal subfield $L$, together with some element~$z$ inducing the
automorphism generating $\Gal(L/F)$, satisfying $z^n = a.$
We begin with some lemmas.

\begin{lem}\label{bas0} There is a field $K(t) \supset Z$ and a degree $n$
cyclic extension $L/K(t)$ such that $L/F$ is rational, and
$UD(F,n) \otimes_Z K(t) = \Delta(L/K(t),t)$.\end{lem}

\begin{proof} Write $Z = F(X \oplus Y)^{S_n}$ as usual (see [Sa] p. 322). 
Let $C_n \subset S_n$ be generated by the $n$ cycle $(1,2, \dots, n)$. Over
$C_n$, $Y \cong M \oplus \Z$ where $M$ is a free $C_n$ lattice. We
can set $L = F(X \oplus M)$, $K = L^{C_n}$, and take $t$ to be the
generator of $\Z$.~\end{proof}

\begin{lem}\label{bas2} Suppose $F$ is a field and $D/F$ is an
division algebra.
Set $$A = D \otimes_F UD(F,n)^b = (D \otimes_F Z) \otimes_Z
UD(F,n)^b.$$ Then  $\index (A)$ is the degree of $D$ times
$n/(n,b) = \index (UD(F,n)^b)$.
\end{lem}\begin{proof}  Since $UD(F,n)$ has   index equal exponent, the  index
of any power is equal to the exponent. In fact, if $b = b'(n,b)$
and $n = n'(n,b)$ then $(b',n') = 1$. Thus $UD(F,n)^b =
(UD(F,n)^{(n,b)})^{b'}$ and the index and exponent of $UD(F,n)^b$
is the same as that of $(UD(F,n)^{(n,b)})$. That is, we may assume
$b|n$.

By Lemma \ref{bas0}, there is a field $K(t) \supset Z$ and a
degree $n$ cyclic extension $L/K(t)$ such that $D \otimes_F L$ is
a division algebra and $UD(F,n) \otimes_Z K(t) =
\Delta(L/K(t),t)$. Of course, by Galois theory,
$\Delta(L/K(t),t)^b$ is equal in the Brauer group to
$\Delta(L'/K(t),t)$ where $L/L'$ has degree~$b$. Finally, $(D
\otimes_F K(t)) \otimes_{K(t)} \Delta(L'/K(t),t)$ is a division
algebra via twisted polynomial rings.\end{proof}

We are in the game of embedding division algebras into bigger division
algebras. The key method is the following.

\begin{thm}\label{thm4.5} Suppose $D/K$ is a division algebra of degree $a$ and $K/F$
has degree~$b$. Assume $E/F$ is a division algebra of degree $N =
nab$. Then $D$ is isomorphic to a subalgebra of $E$ over $F$ if
and only if $(E \otimes_F K) \otimes_K D^{\circ}$ has (Schur)
index dividing~$n$. Furthermore, if this index divides $n$ then it
is equal to $n$. \end{thm}

\begin{proof} Suppose $D \subset E$. In particular, $K \subset E$ and
so $E \otimes_F K$ has   index $N/b$ and we set $E'/K$ to be the
associated division algebra which is the centralizer of $K$ in
$E$. Then $D \subset E'$ and we take $D'$ to be its centralizer,
implying $E' = D \otimes_K D'$. Since $D'$ has degree $n$, we have
proven one direction.

Conversely, suppose     $\index ((E \otimes_F K) \otimes _K
D^{\mathop{op}})$ divides~$n$. Then   $\index (E \otimes_F K)$
divides~$na$, implying $\index (E \otimes_F K)= na$ since $\index
(E \otimes_F K) \ge \frac N b = na.$ Thus, $(E \otimes_F K)
\otimes _K D^{\mathop{op}} \sim A,$ where $A/K$ has degree $n,$
implying
 $[E \otimes_F K]$ is equal in the
Brauer group to $[D][A]$.

 Let $E'$ be the centralizer of $K$ in $E$. Since the degrees
agree, $E' \cong D \otimes_K A$.\end{proof}

\medskip
We are going to force one algebra inside another by using partial
splitting fields and Weil transfers. More specifically, let $A/K$
be a central simple algebra and $n$ an integer dividing the degree
of $A/K$. Let $V_n(A)$ be the variety of rank $n$ left ideals of $A$
and let $K_n(A)$ be its field of fractions. Then for any field $K'
\supset K$, $V_n(A)$ has a $K'$ point if and only if $A \otimes_K
K'$ has index dividing $n$.

Next we set $W_n(A)$ to be the Weil transfer to $F$ of $V_n(A)$,
so for $F' \supset F$, $W_n(A)$ has an $F'$ point if and only if
$V_n(A)$ has an $K \otimes_F F'$ point (and in fact there is a
natural correspondence). Let $F_n(A)$ denote the field of
fractions of $W_n(A)$. Then   $\index(A \otimes _F KF_n(A))$
 divides $n$.

The important tool for using this construction is the following
result (\cite{Sa}) about index reduction, for which we need to
introduce more notation. Let $K/F$ be finite separable with Galois
closure $\bar K/F$. Let $G$ be the Galois group of $\bar K/F$ and
$H \subset G$ the subgroup corresponding to $\bar K/K$. If $r$ is
the degree of $A/K$, then we can define an ``action'' of the $G$
module $R = (\Z/r\Z)[G/H]$ as follows. Let $\bar A = A \otimes_K
\bar K$, so $H$ has a natural semilinear action on $\bar A$ and
for any $g \in G$ we can define the $g$ twist $g(\bar A)$. Of
course, for $g \in G$, $gHg^{-1}$ has a natural semilinear action
on $g(\bar A)$.

For $\alpha \in R$, define $H_{\alpha} = \{g \in G | g\alpha =
\alpha\}$. Define $K(\alpha) = \bar K^{H_{\alpha}}$. Write
$$\alpha = \sum n_{gH}\, gH;$$ then the $n_{gH}$ are constant on
$H_{\alpha}$-orbits. Fix a coset $gH$ and set $e = n_{gH}$. Let 
$L \subset H_{\alpha}$ be the stabilizer
of $gH$. Let 
${\mathcal O} = \{g_iH\}$ be the orbit of $H_{\alpha}$ containing $gH$ 
so $e = n_{g_iH}$ for all $i$. 
Then $L$ acts naturally on $g(\bar A)$ and $H_{\alpha}$
acts on $B_{gH}$ which is the tensor product over 
$\bar K$ of $g_i(\bar A)^e$, one for each $g_iH$ in ${\mathcal O}$. 

Now we let $gH$ vary, one for each $H_{\alpha}$ orbit. 
Tensor over $\bar K$ all the $B_{gH}$ defined above  
and call the resulting $\bar K$ algebra $B$. Note that $\bar K$ 
is the center of $B$. 
Define $A^{\alpha}$ to be the  $H_{\alpha}$ invariant
subring of $B$. Then $A^{\alpha}$ has center $K(\alpha)$. 

Finally, for $\alpha = \sum n_{gH}gH$ as above, define
\begin{equation}\label{eq0}|\alpha| = \prod_{gH}\frac{n}{(n,n_{gH})}  .\end{equation}

\begin{thm}\label{thm4} (\cite{Sa} p. 332). Notation as above, suppose $B/F$
is any central simple algebra (over $F$). Then the
 index of $B \otimes_F F_n(A)$ is the gcd of all the integers
$$\Schur (B \otimes_F  A^{\alpha})[K(\alpha):F]|\alpha|,$$
taken over all $\alpha \in R$.\end{thm}

We actually need a double version of the above result. Let us
assume that $K/F$ and $K'/F$ are finite separable with Galois
closures $\bar K/F$ and $\bar K'/F$ and corresponding groups $G
\supset H$ and $G' \supset H'$. For convenience we may assume that
$\bar K/F$ and $\bar K'/F$ are linearly disjoint. Let $A/K$ and
$A'/K'$ be central simple algebras and let $F_{n,n'}(A,A')$ denote
the join of the fields $F_n(A)$ and $F_{n'}(A')$ over $F$. If
$A'/K'$ has degree $r'$ set $R' = (\Z/r'\Z)[G'/H']$ as above. If
$\alpha \in R$ and $\beta \in R'$ set $K({\alpha,\beta}) =
K(\alpha) \otimes_F K'(\beta)$. Finally write $\beta = \sum_{gH'}
m_{gH'}gH' \in R'$ and set $$|\beta| =
\prod_{gH'}\frac{n'}{(n',m_{gH'})}.$$

\begin{thm}\label{thm5} Suppose $B/F$ is a central simple algebra
and set $$B(\alpha,\beta) = B \otimes_F K(\alpha,\beta).$$ 
Then the (Schur) index
  $\mathbf i : = \index(B \otimes_F F_{n,n'}(A,A'))$ is the 
gcd of the integers
$$\Schur\left((B(\alpha,\beta) \otimes_{K(\alpha,\beta)} 
(A^{\alpha}  \otimes_{K{(\alpha)}} K(\alpha,\beta))  \otimes_{K(\alpha,\beta)}
(A'^{\beta} \otimes_{K'(\beta)} K({\alpha,\beta}))\right)[K({\alpha,\beta}):F]|\alpha||\beta|,$$
ranging over all $\alpha \in R$ and $\beta \in R'$.
\end{thm}

\proof The basic idea here is to apply Theorem~\ref{thm4} twice,
noting that $[K({\alpha,\beta}):F] = [K(\alpha):F][K'(\beta):F]$.
Put $B' = B \otimes_F F_n(A)$. Then, by Theorem~\ref{thm4}, 
using the fact that $[F_nK'(\beta):F_n(A)] = [K'(\beta):F]$, 
$\mathbf i$ is the gcd of all integers
$$\Schur\left( (B' \otimes_{F_n(A)} F_nK'(\beta))  \otimes_{F_nK'(\beta)} ({A'}^\beta  \otimes_{K'(\beta)}
F_nK'(\beta))\right)[K'(\beta):F]|\beta|,$$ where $F_nK'(\beta)$
is the join of $F_n(A)$ and $K'(\beta)$ over $F$. Note that $F_nK'(\beta)$
is the function field of $W_n(A \otimes_F K(\beta))$ which is the
$K({\alpha,\beta})/K'(\beta)$ transfer of $V_n(A \otimes_F
K'(\beta))$. Now by Theorem~\ref{thm4} again, using 
$B'$ instead of~$B$, each 
$$\Schur\left((B' \otimes_{F_n(A)} F_nK'(\beta)) \otimes_{F_nK'(\beta)} 
(A'^{\beta} \otimes_{K'(\beta)} F_nK'(\beta) )\right)$$ 
is the gcd of
$$\Schur\left((B \otimes_F A^{\alpha}) \otimes_{K{(\alpha)}} ( A'^{\beta} \otimes_{F_n K'(\beta)} K({\alpha,\beta}))\right)
[K({\alpha,\beta}):K'(\beta)]|\alpha|$$ and the result follows.

Now suppose we already know that  $D_1/K_1$ and $D_2/K_2$ are
division algebras of degrees $d_i$ and $K_i/F$ is separable of
degree $e_i$. Set $m_i = d_ie_i$ and let $N$ be any multiple of
the lcm of $m_1^2$ and $m_2^2$. Recall that $F'/F$ is called a 
regular field extension if $F'$ is finitely generated as a field 
over $F$, $F'$ is a finite separable extension of a purely 
transcendental extension of $F$, and $F$ is algebraically closed in $F'$. 

\begin{thm}\label{thm6} There is a regular field extension $F' \supset
F$ and a division algebra $E'/F'$ of degree $N$ such that $D_i
\subset E'$ is compatible with $F \subset F'$ for $i = 1,2$.\end{thm}

\proof Set $n_i = N/m_i$, noting that $n_i$ is a multiple of
$m_i$. Set $$E = UD(F,N)$$ and we will extend $Z$ so that
the $D_i$ embed in the base extension of $E$. To achieve this we
set $F' = F_{n_1,n_2}(A_1,A_2)$ where
$$A_i = (D_i^{\circ} \otimes_{K_i} K_iZ) \otimes_{K_iZ} (E \otimes_Z K_iZ).$$
Note that $(E \otimes_Z K_iZ)$ is just the generic division
algebra over $K_i$. Also note that $E' = E \otimes_{Z} F'$ has
both $D_i$ embedded because we have suitably reduced the index of
both $A_i$. The problem is to show that $E'$ is a division
algebra, i.e.,~that $\index(E') = N,$ and for this we apply
Theorem~\ref{thm5}. In applying this theorem, note that the degree
of~$A_1$ is $Nd_1$, and so is a multiple of $n_1$. We make a
similar comment about the degree of $A_2$.

To apply   Theorem~\ref{thm5}, we need to get a handle on
$$\Schur((E \otimes_Z {ZK({\alpha,\beta})})^{1 + \alpha + \beta} \otimes
(D_1^{-\alpha} \otimes_{K(\alpha)} ZK({\alpha,\beta})) \otimes
(D_2^{-\beta} \otimes_{K'(\beta)} ZK({\alpha,\beta}))),$$ where
the unsubscripted tensors are over ${ZK({\alpha,\beta})}$. Write
$\alpha = \sum n_{gH}gH$ and set $a = \sum n_{gH}$ and similarly
for $\beta$ and $b$. Note that $E$ is not moved by either Galois
group so $(E \otimes_Z K({\alpha,\beta})Z)^{1 + \alpha + \beta}$
is $E^{1 + a + b} \otimes_Z K({\alpha,\beta})Z$ which has index
$N/(N,1 + a + b)$ which we define to be $N_{a,b}$. By Lemma~\ref{bas2} the above index is
$$\Schur\left((D_1^{-\alpha} \otimes_{K(\alpha)} K({\alpha,\beta}))\right)
\otimes_{K({\alpha,\beta})} (D_2^{-\beta} \otimes_{K'(\beta)}
K({\alpha,\beta}))N_{a,b}$$ 
and so we want to show that $N$ divides  the
expression
\begin{equation}\label{eq2}\Schur \left((D_1^{\alpha} \otimes_{K(\alpha)} K({\alpha,\beta})\right)
\otimes_{K({\alpha,\beta})} \big(D_2^{\beta} \otimes_{K'(\beta)}
K({\alpha,\beta})\big)[K(\alpha,\beta):F] |\alpha||\beta|N_{a,b}
.
\end{equation}

We show the needed divisibility prime by prime. So assume, for $p$
prime that $p^s$ divides $N$ exactly, in the sense that $\frac
N{p^s}$ is prime to $p$. Likewise, assume that $p^{t_i}$ divides
$n_i$ exactly. Since $N = n_id_ie_i$ and $n_i$ is a multiple of
$d_ie_i$ we have  $2t_i \geq s$ and also $t_1 + t_2 \geq
s$. If $1+a+b$ is prime to $p$ we are done. Thus we assume $p$
divides
\begin{equation}\label{eq3} 1 + \sum_{gH} n_{gH} + \sum_{gH'}
m_{gH'}\end{equation}  and this implies that at least one summand
in 
\begin{equation}\label{eq3.5} \sum_{gH} n_{gh} + \sum_{gH'} m_{gH'}
\end{equation}
is prime to $p$. If any term in~\eqref{eq3.5}, say $n_{gH}$, is prime
to $p$, then $\frac n{(n, n_{gH})}$ is divisible by
$p^{t_1}$. Thus, if two terms in~\eqref{eq3.5} are  prime to
$p$, then $p^{2t_1}$ or $p^{t_1 + t_2}$ or $p^{2t_2}$ divide
$|\alpha||\beta|$ and again we are  done.

Thus we assume that $p$ is prime to exactly one summand in~\eqref{eq3.5}.
Replacing $\alpha$ by $g^{-1}\alpha$,   we  assume that only $n_H$
 is prime to $p$. It follows that $H_{\alpha}$ fixes the trivial coset~$H$ and so
$H_{\alpha} \subseteq H$, implying $K(\alpha) \supseteq K$. Set
$\alpha' = \alpha - n_HH;$ thus, $$|\alpha'| = \prod_{gH \not= H}
\frac n {(n,n_{gH})}.$$ 

We know that $K(\alpha) =
[K(\alpha):K][K:F]$. Write $s = s_1 + s_2 + s_3$ where $p^{s_1}$
is the exact power of $p$ dividing $n$, $p^{s_2}$ is the exact
power dividing $d_1$, and $p^{s_3}$ is the exact power dividing
$e_1$. Note that $p^{s_1}$ divides $\frac n{(n,n_H)}$, and of course
$p^{s_3}$ divides $[K:F]$. Thus it suffices to show that 
$p^{s_2}$ divides
\begin{equation}\label{eq4} \Schur(D'')[K({\alpha,\beta}):K]|\alpha'||\beta|, \end{equation}
where
$$D'' = (D_1^{\alpha} \otimes_{K(\alpha)} K({\alpha,\beta}))
\otimes_{K({\alpha,\beta})} (D_2^{\beta} \otimes_{K'(\beta)}
K({\alpha,\beta})).$$ We will prove in fact that~\eqref{eq4} is divisible
by $d_1$. Note that
$$(D_1^{\alpha} \otimes_{K(\alpha)} K({\alpha,\beta})) =
(D_1 \otimes_K K({\alpha,\beta})) \otimes_{K({\alpha,\beta})}
(D_1^{\alpha'} \otimes_{K(\alpha)} K({\alpha,\beta})).$$

 We need to
estimate  some indices. Of course $D_1$ has index $d_1$, and so
over~$\bar K\bar K'$, $g(D_1)^{n_{gH}}$ has index dividing
$\frac{d_1}{(d_1,n_{gH})}$. Then $D_1^{\alpha'}
\otimes_{K(\alpha)} K({\alpha,\beta})$ has index dividing
$\prod_{gH \not= H} \frac{d_1}{(d_1,n_{gH})}$. Similarly
$D_2^{\beta} \otimes_{K'(\beta)} K({\alpha,\beta})$ has index
dividing $\prod_{gH'} \frac{d_2}{(d_2,m_{gH'})}$.

We need a trivial lemma.

\begin{lem}\label{triv} Suppose $a$ divides $b$. Then $a/(a,d)$ divides
$b/(b,d)$.\end{lem}

\begin{proof} For any prime $p$, the power of
$p$ dividing $\frac a{(a,d)}$ is less than or equal to the power
of $p$
dividing $\frac b{(b,d)}.$

Let $p^{u_1}$ be the exact power of $p$ dividing $\Schur(D_1 \otimes_{K} K({\alpha,\beta}))$ and let 
$$D^{\#} =
(D_1^{\alpha'} \otimes_{K(\alpha)} K({\alpha,\beta}))
\otimes_{K({\alpha,\beta})} (D_2^{\beta} \otimes_{K'(\beta)}
K({\alpha,\beta})).$$ 
Also let $p^{u_2}$ be the exact power of $p$ dividing $\index(D^{\#})$. 
It follows from 
Lemma~\ref{triv} that $p^{u_2}$ divides $|\alpha'||\beta|$. Also, 
$D'' = (D_1^{n_H} 
\otimes_{K} K({\alpha,\beta})) \otimes_{K({\alpha,\beta})}
D^{\#}$. Now let $p^{u_3}$ be the exact power of $p$ 
dividing 
$$\Schur(D_1^{n_H} \otimes_{K}
K({\alpha,\beta})).$$ 
This is the same 
as the exact power of $p$ dividing  
$$\Schur(D_1 \otimes_{K}
K({\alpha,\beta})).$$ 
Set $p^{u_4}$ to be the exact power 
of $p$ dividing $[K(\alpha,\beta):K]$. 
Then $\index(D'')$ is a multiple of $p^{u_1-u_2}$. Thus
$\Schur(D'')|\alpha'||\beta|$ is a multiple of~$p^{u_1}$ and
$$\Schur(D'')|\alpha'||\beta|[K({\alpha,\beta}):K]$$ 
is a multiple
of $p^{u_1 + u_4}$ which is the exact power of $p$ dividing    
$$\Schur\left(D_1 \otimes_{K}
K({\alpha,\beta})\right)[K({\alpha,\beta}):K],$$
 a multiple of $d_1$. This proves Theorem~\ref{thm6}. \end{proof}

Finally let $D'_i/K'_i$ be arbitrary such that both $K_i$ are
finitely generated over a prime or algebraically closed field
$k'$. Let $k$ be the subfield of elements of $K_1$ and $K_2$
algebraic over $k'$. Then $k/k'$ is a finite field extension. In
particular, $k$ is perfect. Then we can write $K'_i \supset F_i$
such that $K_i/F_i$ is finite separable and $F_i/k$ is rational.
Let $F = q(F_1 \otimes_k F_2)$, $K_i = K_i' \otimes_{F_i} F$ and
$D_i = D_i' \otimes_{K_i'} K_i$. We apply Theorem~\ref{thm6} to
the $D_i$ and conclude:

\begin{thm}\label{thm7} Suppose $D_i'/K_i'$ are division algebras with the
$K_i'$ finitely generated over a common field $k'$ which is either
algebraically closed or the prime field. Then the $D_i'$ can be
embedded into a common division algebra $E$ finite over its
center.\end{thm}

The question remains: What is the lowest possible bound for $\deg
E$?



\begin{thebibliography}{KraLen2000}
\bibitem[C]{C} Cohn, P.M. \textit{The embedding of firs in skewfields}, Proc. London
Math. Soc. (3) 23 (1971), 193-213.

\bibitem[FSS]{FSS} Fein, B., Saltman, D., and Schacher, M. \textit{Embedding problems for finite dimensional division algebras}, J.~Algebra 167 (1994), 588--626.

\bibitem[Sa]{Sa} Saltman, D., \textit{The Schur Index and Moody's Theorem},
K-Theory {\bf 7}: 309--332, 1993

\end{thebibliography}
\end{document}